\documentclass[a4paper,12pt]{amsart}

\usepackage{amssymb} 
\usepackage{latexsym} 
\usepackage{amsfonts} 
\usepackage{amsmath} 
\usepackage{eucal} 
\usepackage{bm} 
\usepackage{bbm} 
\usepackage{graphicx} 
\usepackage[english]{varioref} 
\usepackage[nice]{nicefrac} 
\usepackage[all]{xy}

\def\ge{\geqslant}
\def\le{\leqslant}
\def\a{\alpha}

\def\G{\Gamma}

\def\D{\Delta}

\def\o{\omega}
\def\p{\pi}

\def\s{\sigma}

\def\l{\lambda}

\def\i{^{-1}}
\def\cn{\mathcal N}

\newcommand{\supp}{\text {\rm supp}}
\newcommand{\End}{\text {\rm End}}

\newcommand{\Hom}{\text {\rm Hom}}

\newcommand{\mc}{\mathcal}
\newcommand{\OO}{\mc{O}}
\newcommand{\PP}{\mathbb{P}}

\usepackage{amsthm}

\theoremstyle{plain}
\newtheorem{thm}{Theorem}[section] 
\newtheorem*{thm*}{Theorem} 
 
 \newtheorem{lem}[thm]{Lemma}
 \newtheorem{cor}[thm]{Corollary}

\theoremstyle{definition}

\theoremstyle{remark}

\newtheorem*{remark*}{Remark}
\newtheorem*{conj*}{Conjecture(*)}
\newtheorem*{claim*}{Claim}
\newtheorem*{fact1*}{Fact $(a)$}
\newtheorem*{fact2*}{Fact $(b)$}
\newtheorem*{fact3*}{Fact $(c)$}

\begin{document}
 
\title[]{Semi-stable locus of a group compactification} 
\author{Xuhua He}
\address{Department of Mathematics, The Hong Kong University of Science and Technology, Clear Water Bay, Kowloon, Hong Kong}
\email{maxhhe@ust.hk}
\author{Jason Starr}
\address{Department of Mathematics, Stony Brook University, Stony Brook, NY 11794, USA}
\email{jstarr@math.sunysb.edu}
\thanks{X. H. was partially supported by (USA) NSF grant  DMS 0700589 and (HK) RGC grant DAG08/09.SC03. J. S. was partially supported by an Alfred P. Sloan fellowship, NSF grant DMS-0553921 and NSF grant DMS-0758521.}

\begin{abstract}
In this paper, we consider the diagonal action of a connected semisimple group of adjoint type on its wonderful compactification. We show that the semi-stable locus is a union of the $G$-stable pieces and we calculate the geometric quotient. 
\end{abstract} 

\maketitle
 
\subsection{Introduction} 
Let $G$ be a connected, semisimple algebraic group of adjoint type over an algebraically closed field and $X$ be its wonderful compactification. We will give an explicit description of the semi-stable locus of $X$ (for the diagonal $G$-action) using Lusztig's $G$-stable pieces and calculate the geometric quotient $X//G$. We also deal with the case where the $G$-action is twisted by a diagram automorphism. 

The results will be used by the first author \cite{H3} to study character sheaves on the wonderful compactification. 

During the time the article was writing, we learned that De Concini, Kanna and Maffei \cite{DKM} described the semi-stable locus and geometric quotient for complete symmetric varieties (which includes as a special case the non-twisted conjugation action of $G$ on its wonderful compactification).  
 
\subsection{Geometric invariant theory}
The foundations of geometric invariant theory are developed in
\cite{GIT}.  We quickly review that part which we use.  Let $k$ be a
field.  The setup for geometric invariant theory over $k$ consists of
$(G,X,\tau,\mc{L},\psi)$ where
\begin{enumerate}
\item[(i)] 
$G$ is a reductive algebraic group over $k$, 
\item[(ii)]
$X$ is a separated, finite type $k$-scheme, 
\item[(iii)]
$\tau:G \times X \rightarrow X$ is an algebraic
action of $G$ on $X$, 
\item[(iv)]
$\mc{L}$ is an invertible sheaf on $X$, and
\item[(v)]
$\psi:\tau^* \mc{L} \rightarrow \text{pr}_X^* \mc{L}$ is a
  \emph{$G$-linearization} of $\mc{L}$ (where $\text{pr}_X: G \times X \to X$ is the projection), i.e.,
an
isomorphism of invertible sheaves on $G\times X$ which defines a
lifting of the action $\tau$ to an action of $G$ on $\text{Spec}_X
\text{Sym}^\bullet(\mc{L})$.
\end{enumerate}

The fundamental theorem of geometric invariant theory, \cite[Theorem
  1.10, p. 38]{GIT}, associates to this datum a pair
$(X^\text{ss}(\mc{L}), \phi)$.  Here
$X^\text{ss}(\mc{L})$ is the union $X_s$ over all positive integers $n$
and all $G$-invariant sections $s$ of $\Gamma(X,\mc{L}^{\otimes n})$, provided $X_s$ is affine (recall, $X_s$ is defined to be the
maximal open subscheme of $X$ on which $s$ is a generator of
$\mc{L}^{\otimes n}$).  
And $\phi$ is a $G$-invariant $k$-morphism
$$
\phi: X^\text{ss}(\mc{L}) \rightarrow X//_{\mc{L}} G
$$
which is a \emph{uniform categorical quotient} of the action of $G$
on $X^\text{ss}(\mc{L})$.  Moreover the following hold.
\begin{enumerate}
\item[(i)]
The morphism $\phi$ is affine and universally submersive.
\item[(ii)]
For some integer $n > 0$, there exists an ample invertible sheaf
$\mc{M}$ on $X//_{\mc{L}} G$ such that $\phi^*\mc{M}$ is isomorphic to
$\mc{L}^{\otimes n}$ as $G$-linearized invertible sheaves (in
particular, $X//_{\mc{L}} G$ is quasi-projective). 
\item[(iii)]
There exists a unique open subscheme $U$ of $X//_{\mc{L}} G$ such that
$\phi^{-1}(U)$ is the \emph{stable locus}.  And the induced morphism
$\phi:\phi^{-1}(U) \rightarrow U$ is a \emph{uniform geometric
  quotient} of $\phi^{-1}(U)$.
\end{enumerate}

Since we do not make use of them, we will not make precise the
definitions of uniform categorical quotient, stable locus and uniform
geometric quotient.  But we will use a few other known facts about
geometric invariant theory.

\medskip\noindent
\textbf{Fact 1.}
When $X$ is projective and $\mc{L}$ is ample, every open $X_s$ is
affine.  Thus
$X//_{\mc{L}} G$ is canonically isomorphic to
$$
X//_{\mc{L}} G = \text{Proj} \oplus_{n\geq 0} \Gamma(X,\mc{L}^{\otimes
  n})^G
$$
and $X^\text{ss}(\mc{L})$ is the maximal open subscheme of $X$ on
which the natural rational map from $X$ to $X//_{\mc{L}} G$ is
defined, \cite{Ses}.

\medskip\noindent
\textbf{Fact 2.}
Again when $X$ is proper, every $G$-orbit $O$ in $X^\text{ss}(\mc{L})$
contains a unique closed $G$-orbit in its closure (in
$X^\text{ss}(\mc{L})$).  And two $G$-orbits $O_1$ and $O_2$ in
$X^\text{ss}(\mc{L})$ are in the same fiber of $\phi$ if and only if
the associated closed $G$-orbits are equals.
In particular, $\phi$ establishes a natural bijection
between the points of $X//_{\mc{L}} G$ and the closed $G$-orbits in
$X^\text{ss}(\mc{L})$,
\cite{Ses}.

\medskip\noindent
\textbf{Fact 3.}(Matsushima's criterion)
A $G$-orbit $O$ is affine if and only if the stabilizer
group of one (and hence every) closed point is itself reductive,
\cite{Ri}.
In
particular, since the fibers of $\phi$ are affine, every closed
$G$-orbit in $X^\text{ss}(\mc{L})$ is affine, and hence has reductive
stabilizer group.  

\medskip\noindent
\textbf{Fact 4.}
If $X$ is normal (or if $X^\text{ss}(\mc{L})$ is normal), then $\phi$ factors through the normalization of the target. Thus by the universal property , the target $X//_{\mc{L}} G$ is normal. 

\subsection{Notations}
Now we fix the notations used in the rest of this article. 
Let $G$ be a connected semisimple algebraic group of adjoint type over
an algebraically closed field $k$. Let $B$ be a Borel subgroup of $G$,
$B^-$ be an opposite Borel subgroup and $T=B \cap B^-$. Let $(\a_i)_{i
  \in I}$ be the set of simple roots determined by $(B, T)$. We denote
by $W$ the Weyl group $N(T)/T$. For $w \in W$, we choose a
representative $\dot w$ in $N(T)$. For $i \in I$, we denote by $\o_i$
and $s_i$ the fundamental weight and the simple reflection
corresponding to $\a_i$.

For $J \subset I$, let $P_J \supset B$ be the standard parabolic
subgroup defined by $J$ and let $P^-_J \supset B^-$ be the parabolic
subgroup opposite to $P_J$. Set $L_J=P_J \cap P^-_J$. Then $L_J$ is a
Levi subgroup of $P_J$ and $P^-_J$. The semisimple quotient of $L_J$
of adjoint type will be denoted by $G_J$. We denote by $\p_{P_J}$
(resp. $\p_{P^-_J}$) the projection of $P_J$ (resp. $P^-_J$) onto
$G_J$. Let $W_J$ be the subgroup of $W$ generated by $\{s_j \mid j \in
J\}$ and $W^J$ be the set of minimal length coset representatives of
$W/W_J$.

\subsection{Wonderful compactification of $G$}
We consider $G$ as a $G \times G$-variety by left and right
translation. Then there exists a canonical $G \times G$-equivariant
embedding $X$ of $G$ which is called the {\it wonderful
  compactification} (\cite{DP}, \cite{Str}). The variety $X$ is an irreducible,
smooth projective $(G \times G)$-variety with finitely many $G \times
G$-orbits $Z_J$ indexed by the subsets $J$ of $I$. The boundary $X-G$
is a union of smooth divisors $\overline{Z_{I-\{i\}}}$ (for $i \in
I$), with normal crossing. The $G \times G$-variety $Z_J$ is
isomorphic to the product $(G \times G) \times_{P^-_J \times P_J}
G_J$, where $P^-_J \times P_J$ acts on $G \times G$ by $(q, p) \cdot
(g_1, g_2)=(g_1 q \i, g_2 p \i)$ and on $G_J$ by $(q, p) \cdot
z=\p_{P^-_J}(q) z \p_{P_J}(p) \i$. We denote by $h_J$ the image of
$(1, 1, 1)$ in $Z_J$ under this isomorphism.

\subsection{Twisted actions}
We follow the approach in \cite[Section 3]{HT}. Let $\s$ be an
automorphism on $G$ such that $\s(B)=B$ and $\s(T)=T$. We also assume
that $\s$ is a diagram automorphism, i.e., the order of $\s$ coincides
with the order of the associated permutation on $I$.

Let $G_{\s}$ (resp. $X_{\s}$) be the $(G \times G)$-variety which as a
variety is isomorphic to $G$ (resp. $X$), but the $G \times G$-action
is twisted by $(g, g') \mapsto (g, \s(g'))$. Then $G_{\s}$ is an open
$G \times G$-subvariety of $X_{\s}$ and we call $X_{\s}$ the wonderful
compactification of $G_{\s}$.

Under the natural bijection between $X$ and $X_{\s}$, we may identify
the $G \times G$-orbits on $X$ with the $G \times G$-orbits on
$X_{\s}$. We denote by $Z_{J, \s}$ the $G \times G$-orbit on $X_{\s}$
that corresponds to $Z_{\s(J)} \subset X$. Accordingly, we denote by
$h_{J, \s}$ the base point in $Z_{J, \s}$ which corresponds to the
base point $h_{\s(J)}$ of $Z_{\s(J)}$.

\subsection{$\s$-semisimple elements in $G_{\s}$}

We follow the notation of \cite{Spr}. An element $g \in G_{\s}$ is
called {\it $\s$-semisimple} if it is conjugated to an element in
$T$. We have the following result.

\begin{thm} \label{thm-1}
Let $g \in G_{\s}$. Then the following conditions are equivalent:

(1) The element $g$ is $\s$-semisimple.

(2) The $G$-orbit of $g$ is closed in $G_{\s}$.

In this case, the isotropy subgroup of $g$ in $G$ is reductive.
\end{thm}

The equivalence of (1) and (2) can be found in \cite[1.4 (e)]{L1} (in terms of disconnected groups instead of twisted conjugation action). In the case of simply connected group, the equivalence is also proved in \cite[Proposition 3]{Spr}. By Fact 3, Matsushima's criterion, the $G$-orbit of $g$ is closed implies that the isotropy subgroup of $g$ is reductive. 

\subsection{$G$-stable-piece decomposition}
Let $G_{\D}$ be the diagonal image of $G$ in $G \times G$. The
classification of the $G_{\D}$-orbits on $X$ was obtained by Lusztig
\cite{L} in terms of $G$-stable pieces. A similar result also occurs
in \cite{EL}.  We list some known results which will be used later.

For $J \subset I$ and $w \in W^{\s(J)}$, set $$Z_{J, \s; w}=G_{\D} (B
\dot w, B) \cdot h_{J, \s}.$$ We call $Z_{J, \s; w}$ a {\it $G$-stable
  piece} of $X_{\s}$. By \cite[12.3]{L} and \cite[Proposition
  2.6]{H1}, $X_{\s}$ is a disjoint union of the $G$-stable pieces.

\medskip\noindent
\textbf{Fact 5.}
$X_{\s}=\bigsqcup_{J \subset I} \bigsqcup_{ w \in W^{\s(J)}} Z_{J, \s;
    w}.$

\medskip\noindent
Set $I(J, \s; w)=\max\{K \subset J \mid w \s(K)=K\}$. Then the
subvariety $L_{I(J, \s; w)} \dot w$ of $G_{\s}$ is stable under the
action of $L_{I(J, \s; w)} \times L_{I(J, \s; w)}$ and in particular,
is stable under the conjugation action of $L_{I(J, \s; w)}$. Moreover,
by \cite[12.3(a)]{L} and \cite[Lemma 1.4]{H2}, $$Z_{J, \s; w}=G_{\D}
(L_{I(J, \s; w)} \dot w, 1) \cdot h_{J, \s}$$ and there exists a
natural bijection between the $G_{\D}$-orbits on $Z_{J, \s; w}$ and
the $L_{I(J, \s; w)}$-orbits on $L_{I(J, \s; w)} \dot w/Z^0(L_J)
\subset G_{\s}/Z^0(L_J)$ (for the conjugation action of $L_{I(J, \s;
  w)}$).

For any point $z$ in $Z_{J, \s; w}$, the isotropy subgroup $$G_z=\{g
\in G \mid (g, g) \cdot z=z\}$$ was described explicitly in
\cite[Theorem 3.13]{EL}. We only need the following special case in
our paper.

\medskip\noindent
\textbf{Fact 6.}
Let $z=(g l \dot w, g) \cdot h_{J, \s}$ for $g \in G$ and $l \in
L_{I(J, \s; w)}$. Then $G_z$ is reductive if and only if $w=1$ and $l$
is a $\s$-semisimple element in $L_{I(J, \s; 1)}$.

\medskip\noindent
By \cite[Theorem 4.5]{H2}, the closure of each $G$-stable piece is a
union of $G$-stable pieces and the closure relation can be described
explicitly. More precisely, for $J \subset I$, $w \in W^{\s(J)}$ and
$w' \in W$, we write $w' \le_{J, \s} w$ if there exists $u \in W_J$
such that $w' \ge u w \s(u) \i$. Then $$\overline{Z_{J, \s;
    w}}=\sqcup_{J' \subset J} \sqcup_{w' \in W^{J'}, w' \le_{J, \s} w}
Z_{J', \s; w'}.$$ Notice that if $1 \le_{J, \s} w$, then we must have
$w=1$. Therefore,

\medskip\noindent
\textbf{Fact 7.}
$\sqcup_{J \subset I} Z_{J, \s; 1}$ is open in $X_{\s}$.

\subsection{Nilpotent Cone of $X$}
For any dominant weight $\l$, let ${\rm H}({\l})$ be the dual Weyl
module for $G_{\rm sc}$ with lowest weight $-\l$. Let $^{\s} {\rm
  H}(\l)$ be the $G_{\rm sc}$-module which as a vector space is ${\rm
  H}(\l)$, but the $G_{\rm sc}$-action is twisted by the automorphism
$\s$ on $G_{\rm sc}$. Then there exists (up to a nonzero constant) a
unique $G_{\rm sc}$ isomorphism $^{\s} {\rm H}(\l) \to {\rm
  H}(\s(\l))$. In particular, if $\l=\s(\l)$, then we have an
isomorphism $f_{\l}: {}^{\s} {\rm H}(\l) \to {\rm H}(\l)$.

By \cite[3.9]{DS}, there exists a $G \times G$-equivariant
morphism $${\rho_{\l}} : X \rightarrow \mathbb P \bigl( \End({\rm
  H}({\l})) \bigr)$$ which extends the morphism $G_\sigma \rightarrow
\mathbb P \bigl(\End({\rm H}({\l})) \bigr)$ defined by $g \mapsto g [
  {\rm Id}_{\l}]$, where $ [{\rm Id}_{\l}]$ denotes the class
representing the identity map on ${\rm H}({\l})$ and $g$ acts by the
left action.  We denote by $\mc{L}_X(\lambda)$ the $G_{\rm sc}\times
G_{\rm sc}$-linearized invertible sheaf on
$X$ which is the pullback under $\rho_\lambda$of $\OO(1)$ with its canonical linearization.  This
is the ``usual'' linearized 
invertible sheaf on $X$ associated to the weight
$\lambda$, e.g., as defined in \cite[p. 100]{BP}.  For sufficiently
divisible and positive $n$, the $G_{\rm sc} \times G_{\rm
  sc}$-linearization of $\mc{L}_X(\lambda)^{\otimes n} =
\mc{L}_X(n\cdot \lambda)$ factors through a $G\times
G$-linearization. This induces a $G_{\D}$-linearization of $\mc L_X(\lambda)^{\otimes n}$. If moreover, $\lambda$ is regular, then $\mc L_X(\lambda)$ is ample (see \cite[section 2]{Str}).

The morphism $\rho_{\l}$ induces a $G \times G$-equivariant morphism
$X_{\s} \to \mathbb P \bigl(\Hom({}^{\s} {\rm H}(\l), {\rm H} (\l))
\bigr)$. When $\l=\s(\l)$, we may apply the isomorphism $f_{\l}:
     {}^{\s} {\rm H}(\l) \to {\rm H}(\l)$ to obtain the $G \times
     G$-equivariant morphism $$\rho_{\l, \s}: X_{\s} \to \mathbb P
     \bigl( \End({\rm H}({\l})) \bigr).$$
As above, $\mc{L}_{X_{\s}}(\lambda,\sigma)$ denotes the $G_{\rm sc}
\times G_{\rm sc}$-linearized invertible sheaf on $X_\sigma$ which is
the pullback  under $\rho_{\lambda,\sigma}$ of $\OO(1)$ with its canonical linearization.  Of course
$X_\sigma$ equals $X$ as varieties, and $\mc{L}_X(\lambda)$ equals
$\mc{L}_{X_{\s}}(\lambda,\sigma)$ as invertible sheaves on this
variety.  But 
the $G\times G$-actions are not the same, and thus the $G\times
G$-linearized invertible sheaves are not the same.  

For $\l=\s(\l)$, let $\cn(\l)_{\s}$ be the subvariety of $X_{\s}$
consisting of elements that may be represented by a nilpotent
endomorphism of ${\rm H}(\l)$. We call $\cn(\l)_{\s}$ the {\it
  nilpotent cone} of $X_{\l}$ associated to the dominant weight
$\l$. We have an explicit description of $\cn(\l)$ which was obtained
in \cite[Proposition 4.4]{HT} $$\cn(\l)_{\s}=\bigsqcup_{J \subset I}
\bigsqcup_{\substack{w \in W^{\s(J)} \\ I(\l) \cap \supp(w) \neq
    \varnothing}} Z_{J, \s; w},$$ where $I(\l)=\{i \in I \mid a_i \neq
0\}$ of $I$ for $\l =\sum_{i \in I} a_i \o_i$ and $\supp(w) \subset I$
is the set of simple roots whose associated simple reflections occur
in some (or equivalently, any) reduced decomposition of $w$.

Two subvarieties of $X$ related to the nilpotent cones of $X$ are of
special interest. One is $$\cap_{\l \text{ is dominant}}
\cn(\l)_{\s}=\sqcup_{J \subset I} \sqcup_{w \in W^{\s(J)}, \supp(w)=I}
Z_{J, \s; w}.$$ This subvariety is actually the boudary of the closure
in $X_{\s}$ of unipotent subvariety of $G_{\s}$ in the case where $G$
is simple (See \cite[Theorem 4.3]{H1} and \cite[Theorem 7.3]{HT}).

The other one is $X_{\s}-\cup_{\l \text{ is dominant}}
\cn(\l)_{\s}=\sqcup_{J \subset I} Z_{J, \s; 1}$, which is the
complement of $\cn(\l)_{\s}$ for any $\s$-stable dominant regular
weight. By the next theorem, this subvariety is actually the
semi-stable locus of $X_{\s}$ for the $G_{\D}$-action.

\begin{thm} \label{thm-2}
For $\lambda$ as above, i.e., $\s$-stable, dominant and regular, the semistable
locus $(X_{\s})^\text{ss}(\mc{L}_X(\lambda)^{\otimes n})$ equals
$\sqcup_{J \subset I} Z_{J, \s; 1}$.  In particular, the semistable locus is independent of the choice of weight $\lambda$.  
\end{thm}

\begin{proof}
We simply write the semistable locus $(X_{\s})^\text{ss}(\mc{L}_X(\lambda)^{\otimes n})$ as $X_{\s}^\text{ss}$.  
On $\text{End}(H(\lambda))$ the
characteristic polynomial map
$$
\chi: \text{End}(H(\lambda)) \rightarrow k[t], \qquad (f: H(\lambda) \to H(\lambda)) \mapsto \chi_f(t)
$$
is a morphism which is
invariant under the conjugation action.  The coefficients of the
characteristic polynomial define homogeneous polynomials on
$\text{End}(H(\lambda))$ which are invariant under the conjugation
action.  Also the degree is positive except for the leading coefficient (which is $1$). Thus each non-leading coefficient defines a $G_{\D}$-invariant sections of positive power
$\OO(n)$ on $\PP(\text{End}(H(\lambda)))$.  The pullbacks of these
sections are $G_{\D}$-invariant sections of positive powers $\mc{L}^{\otimes
  n}$.  By Fact 1, the nonvanishing locus of each of these sections is
in the semistable locus.  Equivalently, the non-semistable locus is
contained in the common zero locus of all of these sections.  But the
common zero locus of these pullback sections on $X_{\s}$ equals the
inverse image of the common zero locus of the original sections on
$\PP(\text{End}(H(\lambda)))$.  And this common zero locus is
precisely the nilpotent cone in $\PP(\text{End}(H(\lambda)))$.  Thus
the non-semistable locus is contained in $\mc{N}(\lambda)_\sigma$.  So
$X_{\s}^\text{ss}$ contains $X_{\s} - \mc{N}(\lambda)_\sigma$, i.e.,
$X_{\s}^\text{ss}$ contains $\sqcup_{J\subset I} Z_{J,\s;1}$.  

Also, by Fact 7, $X_{\s}^{ss}-\sqcup_{J \subset I} Z_{J, \s;
  1}$ is closed in $X_{\s}^{ss}$. If $X_{\s}^{ss}$ strictly contains
$\sqcup_{J
  \subset I} Z_{J, \s; 1}$, then there exists a closed $G_{\D}$-orbit
in $X_{\s}^{ss}$ that is not contained in $\sqcup_{J \subset I} Z_{J,
  \s; 1}$. Let $z$ be an element in that orbit. By Fact 3 above, the
isotropy subgroup of $z$, $\{g \in G \mid (g, g) \cdot z=z\}$, 
is reductive. By
Fact 5 above, $z$ is in $Z_{J, \s; w}$ for some $J
\subset I$ 
and $w \in W^{\s(J)}$ with $w \neq 1$. But this contradicts 
Fact 6 above.  
Therefore $X_{\s}^{ss}$ equals $\sqcup_{J \subset I} Z_{J, \s; 1}$. 
\end{proof}

\

In the rest of this paper, we will describe the geometric quotient $X//G$. The description of the geometric quotient for the twisted action is based on some detailed analysis on the structure of the semi-stable locus and will be included in a future article. 

\begin{lem} \label{lem-3}
A $G_{\D}$-orbit $\mc{O}$ in $X^\text{ss}$ is closed in $X^\text{ss}$
if and only if it intersects $\bar{T}$.  
\end{lem}

\begin{proof}
By \cite[Lemma 6.1.6 (ii)]{BK}, $\bar{T}=\sqcup_{J \subset I} N(T)_{\D} \cdot (T, 1) \cdot h_J$. Let $z \in X^\text{ss}$ such that $G_{\D} \cdot z$ is closed in $X^\text{ss}$. Then $G_{\D} \cdot z$ is affine and $G_z$ is reductive. By Fact 6, $z$ is of the form $\{(g l, g) \cdot h_J \mid g \in G\}$ for some semisimple element $l \in L_J$. Hence $l$ is conjugate to an element in $T$. Therefore $G_{\D} \cdot z \cap \bar{T} \neq \emptyset$.

On the other hand, if $z \in \bar{T}$, then the isotropy subgroup of $z$ in $G$ contains $T$. Let $O$ be the unique closed $G$-orbit that is contained in the closure of $G_{\D} \cdot z$. Then $G_{\D} \cdot z$ and $O$ lie in the same fiber of $\phi$, which is an affine variety. By \cite[Page 70, Corollary 1]{St}, $G_{\D} \cdot z$ is closed in that fiber. Hence $\G_{\D} \cdot z=O$.
\end{proof}

\begin{lem} \label{lem-4}
For every element $z$ in $\bar{T}$, the intersection $G_{\D} \cdot z \cap \bar{T}$ of the $G$-orbit with $\bar{T}$ equals the $N(T)$-orbit
$N(T)_{\D} \cdot z$.
\end{lem}

\begin{proof}
Obviously $N(T)_{\D} \cdot z$ is contained in $G_{\D} \cdot z \cap
\bar{T}$.  The content of the lemma is the opposite inclusion.

We may assume without loss of generality that $z$ has the form $z=(t, 1) \cdot h_J$. Suppose that $(g, g) \cdot z$ equals $(t', 1) \cdot h_J$ for some $t' \in T$, i.e., $(g,g)\cdot z$ is a point of $G_{\D} \cdot z \cap (T, 1) \cdot h_J$. 

Let $F_J=(P_J, P_J) \cdot h_J$, then by \cite[Proposition 1.10]{H2}, the action of $G$ on $X$ induces an isomorphism of $Z_{J, 1}$ with $G \times_{P_J} F_J$. Thus $g$ is in $P_J$.  Also both 
$t$ and $t'$ are contained in the same $P_J$-orbit in $L_J /Z^0(L_J)$,
i.e. in the same $G_J$-conjugacy class. Hence there exists an element
$n$ in $N(T)$ such that $t'$ equals $n t n \i$. 
Therefore
$G_{\D} \cdot z \cap (T, 1) \cdot h_J$ is a subset of $N(T)_{\D} \cdot
z$. Now also
\begin{align*} G_{\D} \cdot z \cap \bar{T} &=G_{\D} \cdot z
  \cap N(T)_{\D} (T, 1) \cdot h_J=N(T)_{\D} \cdot (G_{\D} \cdot
  z \cap (T, 1) \cdot h_J) \\ & \subset N(T)_{\D} \cdot ((N(T)_{\D}
  \cdot z)=N(T)_{\D} \cdot z\end{align*} 
proving the lemma.
\end{proof}

\begin{cor} \label{cor-5}
The embedding $\bar{T} \to X_{\s}$ induces an isomorphism 
$$
i: \bar{T}/W \to X//G.
$$
\end{cor}

\begin{proof}
Notice that $\bar{T}/N(T)$ equals $\bar{T}/W$ for the
natural $W$-action on $\bar{T}$ which extends the $W$-action on
$T$. 

The morphism $\bar{T} \to X \to X//G$ is
$N(T)$-invariant, and hence factors through a morphism
$$
i: \bar{T}/W \to X//G.
$$
By Lemma ~\ref{lem-3},
every closed $G$-orbit in $X^{ss}$ intersects 
$\bar{T}$.  Thus $i$ is surjective. For every element $z$ in
$\bar{T}$, the $G$-orbit of $z$ is closed in $X^\text{ss}$.  Thus
two elements $z, z'$ in $\bar{T}$ have the same image under $i$ if and
only if they lie in the same $G$-orbit.  
On the other hand, by Lemma
~\ref{lem-4}, two 
elements $z, z'$ in $\bar{T}$ lie in the same $G$-orbit if and only if
they lie in the same $N(T)$-orbit. Hence $i$ is a bijection on
points.  

By \cite[section 6]{Ste}, the restriction of $i$ to the open
subvariety $T/W$ of $\bar{T}/W$ gives an isomorphism $T/W \cong
G//G$. Hence, as above, $i$ is a bijective, birational morphism of
varieties whose target is a normal variety.  So again by Zariski's Main Theorem, $i$
is an isomorphism.  
\end{proof}

\nocite{DKM}

\bibliography{semistable}
\bibliographystyle{amsalpha}

\end{document}